\documentclass[10pt]{amsart}
\usepackage{amsmath,amssymb,amsthm}
\usepackage{array}

\def\n{{\mathbf {n}}}
\def\A{\mathcal A}
\def\R{\mathbb R}

\newcommand{\dive}{\hbox{div}}
\newcommand{\N}{\mathbb{N}}
\newcommand{\sob}{W^{1,p}(\Omega)}

\newtheorem{thm}{Theorem}[section]

\newtheorem{lem}[thm]{Lemma}
\newtheorem{rem}[thm]{Remark}

\newtheorem{prop}[thm]{Proposition}

\numberwithin{equation}{section}

\begin{document}
\title[An optimization problem]{An optimization problem for nonlinear Steklov eigenvalues with a boundary potential}
\author[J. Fern\'andez Bonder]{Juli\'an Fern\'andez Bonder}

\address[J. Fern\'andez Bonder]{IMAS - CONICET and Departamento de Matem\'atica, FCEyN - Universidad de Buenos Aires, Ciudad Universitaria, Pabell\'on I  (1428) Buenos Aires, Argentina.}

 \email{jfbonder@dm.uba.ar}

\urladdr{http://mate.dm.uba.ar/~jfbonder}

\author[G.O. Giubergia]{Graciela O. Giubergia}
\address[G.O. Giubergia]{Departamento de Matematicas, Universidad Nacional de Rio Cuarto, 5800 Rio Cuarto, Argentina.}

\email{ggiubergia@exa.unrc.edu.ar}

\author[F.D. Mazzone]{ Fernando D. Mazzone }

\address[F.D. Mazzone]{Departamento de Matematicas, Universidad Nacional de Rio Cuarto, 5800 Rio Cuarto, Argentina.}

\email{fmazzone@exa.unrc.edu.ar}



 \maketitle

\begin{abstract}
In this paper, we analyze an optimization problem for the first (nonlinear) Steklov eigenvalue plus a boundary potential with respect to the potential function which is assumed to be uniformly bounded and with fixed $L^1$-norm.
\end{abstract}

\setlength{\extrarowheight}{-5mm}

\section{Introduction }
In recent years a great deal of attention has been putted in optimal design problems for eigenvalues (both linear and nonlinear) due to many interesting applications. For a comprehensive description of the current developments in the field in the case of linear eigenvalues and very interesting open problems, we refer to \cite{Henrot-Pierre}. In the nonlinear setting, we refer to the recent research papers \cite{Cuccu1, Cuccu3, Cuccu2, DPFBR, DP, FBRW} and references therein.

To be precise, the eigenvalue problem that we are interested in is the following
\begin{equation}\label{pde}
\begin{cases}
-\Delta_p u+|u|^{p-2}u=0 &\text{in }\Omega,  \\
|\nabla u|^{p-2}\frac{\partial u}{\partial \n}+\sigma \phi |u|^{p-2}u=\lambda |u|^{p-2}u & \text{in } \partial\Omega.
\end{cases}
\end{equation}
Here $\Omega\subset\R^n$ is a bounded smooth domain, $\Delta_p u$ is the usual $p$-Laplace operator defined as $\Delta_p u = \dive\big( |\nabla u|^{p-2}\nabla u\big)$, $\n$ denotes the outer unit normal vector to $\partial\Omega$, $\phi\in L^\infty(\partial\Omega)$ is a nonnegative boundary potential and $\sigma>0$ is a real parameter.

Under these hypotheses, the functional associated to \eqref{pde} is trivially coercive, that is
\[ I(u,\phi)=\int_{\Omega}|\nabla u|^p+|u|^pdx+\sigma\int_{\partial \Omega}\phi|u|^pd\mathcal{H}^{n-1}\ge \|u\|_{\sob}^p.\]

This functional is associated to \eqref{pde} in the sense that eigenvalues $\lambda$ of \eqref{pde} are critical values of $I$ restricted to the manifold $\|u\|_{L^p(\partial\Omega)}=1$. See \cite{FBR}.

In particular, It is easy to see that the minimum value of $I$
\begin{equation}\label{var prob}
    \lambda(\sigma,\phi)=\inf \big\{I(u,\phi):u\in W^{1,p}(\Omega),\|u\|_{L^p(\partial\Omega)}=1\big\}
\end{equation}
is the first (lowest) eigenvalue of \eqref{pde}. Therefore, the existence of the first eigenvalue and the corresponding eigenfunction $u$ follows from the compact embedding $W^{1,p}(\Omega)\subset L^{p}(\partial \Omega)$. 

In this work, we are interested in the minimization problem for $\lambda(\sigma, \phi)$ with respect to different configurations for the boundary potential $\phi$. That is, given certain class of admissible potentials $\mathcal A$, we look for the minimum possible value of $\lambda(\sigma, \phi)$ when $\phi\in \mathcal A$.

This study complements the ones started in \cite{DPFBR}. In that paper, the authors analyzed the Steklov problem but with an interior potential and show the connections of that problem with the one considered in \cite{FBRW}.

In this opportunity, we consider the class of uniformly bounded potentials, i.e.
$$
\mathcal A = \{\phi\in L^{\infty}(\partial\Omega)\colon 0\le \phi\le1\}.
$$
Observe that $\mathcal A$ is the closure of the characteristic functions in the weak* topology.

Clearly, the minimization problem in the whole class $\mathcal A$ has no sense since the infimum is realized with $\phi\equiv 0$. The relevant problem here is to consider the minimization among those potentials in $\mathcal A$ that has fixed $L^1-$norm. That is
\begin{equation}\label{optimo} 
\Lambda (\sigma, a)= \inf \bigg\{\lambda(\sigma,\phi)\colon \phi\in \mathcal A, \int_{\partial \Omega} \phi\, d\mathcal{H}^{n-1}=a \bigg\}\end{equation}

The first result in this paper is the existence of an optimal potential for $\Lambda(\sigma, a)$ and, moreover, it is shown that this optimal potential can be taken as the characteristic function a sub-level set $D_\sigma$ of the corresponding eigenfunction. See \cite{Chanillo1, Chanillo2} for related results.

As another application we investigate the connection with the optimization problem considered in \cite{DPFBN}. That is,  given $E \subset \partial\Omega$, consider the equation
\begin{equation}\label{agujeroe}
    \left\{%
\begin{array}{rl}
    -\Delta_p u+|u|^{p-2}u&=0\quad\hbox{in }\Omega  \\
    u&=0\quad\hbox {in } E\\
    |\nabla u|^{p-2}\frac{\partial u}{\partial \nu}&=\lambda |u|^{p-2}u \quad\hbox{in } \partial\Omega\setminus E\\
\end{array}%
\right.
\end{equation}
whose  first eigenvalue is given by
\begin{equation}\label{agujero}
    \lambda(\infty ,E):=\inf\big\{ \|u\|^p_{W^{1,p}(\Omega)}\colon  \|u\|_{L^p(\partial\Omega)}=1, u=0, \mathcal{H}^{n-1} \hbox{ a.e. in }
    E\big\},
\end{equation}

Associated to \eqref{agujero} we have  the optimal configuration problem
\begin{equation}\label{agujerooptimo}
    \Lambda(\infty,a)=\inf\big\{ \lambda(\infty,E): \mathcal{H}^{n-1}(E)=a\}.
\end{equation}

Our second result shows that $\Lambda(\sigma, a)\to \Lambda(\infty, a)$ as $\sigma\to\infty$ and, moreover, the optimal configuration $\phi_\sigma=\chi_{D_\sigma}$ of $\Lambda(\sigma, a)$ converges (in the topology of $L^1-$convergence of the characteristic functions) to an optimal configuration of the limit problem $\Lambda(\infty, a)$.

The remaining of the paper is devoted to analyze qualitative properties of optimal configurations for $\Lambda(\sigma, a)$.

First, we consider the spherical symmetric case, that is when $\Omega$ is a ball, and in this simple case by means of symmetrization arguments we can give a full description of the optimal configurations.

Finally, we address the general problem and study the behavior of $\lambda(\sigma, \chi_D)$ for regular deformations of the set $D$. We employ the so--called method of Hadamard and prove differentiability of $\lambda(\sigma, \chi_D)$ with respect to regular deformations and provide a simple formula for the derivative of the eigenvalue. The main novelty of this formula is that it involves a $(n-2)-$dimensional integral along the boundary of $D$ relative to $\partial\Omega$. Up to our knowledge, this is the first time that this type of lower-dimensional integrals were observed in this type of computations.

We want to remark that the results in this work are new even in the linear setting, $p=2$.

\section{Preliminary remarks}

A simple modification of the arguments in \cite{FBR2} shows that, given $\phi\in\A$ and $\sigma>0$, the first eigenvalue $\lambda(\sigma,\phi)$ is simple. i.e. any two eigenfunctions are multiple of each other. Therefore, there exists a unique nonnegative, normalized eigenfunction $u$ (normalized means that $\|u\|_{L^p(\partial\Omega)}=1$).

The purpose of this very short section is to recall some regularity properties of this eigenfunction.

First, we note that by \cite{Serrin}, there exists $\alpha>0$ such that $u \in C^{1,\alpha}_{loc}(\Omega)$.  Now, by an usual argument, we have that $|u|$ is an eigenfunction associated to $\lambda(\sigma,\phi)$. Hence, the Harnack inequallity, c.f. \cite{Serrin}, implies that any first eigenfunction $u$ has constant sign and, moreover, that $u>0$ in $\Omega$. 

Next, by the results of \cite{Lieberman}, an eigenfunction of \eqref{pde} is continuous up to the boundary. In fact, $u\in C^{\beta}(\bar \Omega)$ for some $\beta>0$.

Summing up, we have
\begin{prop}\label{regularidad}
Given $\phi\in\A$ and $\sigma>0$, there exists a unique nonnegative eigenfunction $u\in W^{1,p}(\Omega)$ of \eqref{pde} associated to $\lambda(\sigma, \phi)$. Moreover, this eigenfunction $u$ verifies that $u\in C^{1,\alpha}_{loc}(\Omega)\cap C^{\beta}(\bar\Omega)$ for some $\alpha, \beta>0$. Finally, $u>0$ in $\Omega$.
\end{prop}

\section{Existence of optimal configurations}

In this section we first establish the existence of optimal configurations for $\Lambda(\sigma, a)$. Then we analyze the limit $\sigma\to\infty$ and show the convergence to the problem $\Lambda(\infty, a)$.

Let us begin with the existence result.

\begin{thm}
For any $\sigma >0$ and $0 \leq a \leq \mathcal{H}^{n-1}(\partial \Omega)$ there exist an optimal pair $(u, \phi)\in \sob\times \mathcal A$, which has the following
properties
\begin{enumerate}
\item $u \in C^{1,\alpha}_{loc} (\Omega)\cap C(\bar \Omega)$
\item $\phi=\chi_{D}$ where, for some $s$, $\{u < s\} \subset D \subset \{u \leq s\}$, $\mathcal{H}^{n-1}(D)=a$
\end{enumerate}
\end{thm}
\begin{proof}
We consider a minimizing sequence $\{\phi_k\}_{k\in\N}\subset \mathcal A$ of \eqref{optimo} and their associated normalized eigenfunctions $\{u_k\}_{k\in\N}\subset \sob$.

From the reflexivity of the Sobolev space $W^{1,p} (\Omega)$, the compactness of the embeddings $W^{1,p}(\Omega)\hookrightarrow L^{p}(\partial \Omega)$ and $W^{1,p}(\Omega)\hookrightarrow L^{p}(\Omega)$ and $L^{\infty} (\partial \Omega)$ being a dual space, we obtain a subsequence (again denoted $\{u_k, \phi_k\}$) and $u \in W^{1,p} (\Omega)$, $\phi \in L^{\infty}(\partial \Omega)$ such that
\begin{align}
\label{debil} u_k\rightharpoonup u &\quad\hbox{in }W^{1,p} (\Omega)\\
\label{fuerteborde} u_k\rightarrow u &\quad\hbox{in }L^{p} (\partial \Omega)\\
\label{fuerte} u_k\rightarrow u &\quad\hbox{in }L^{p} (\Omega)\\
\label{debil*} \phi_k\overset{*}{\rightharpoonup} \phi
&\quad\hbox{in }L^{\infty} (\partial \Omega)
\end{align}
From the admissibility of $\phi_k$ and \eqref{debil*}, we get $ 0 \leq \phi \leq 1$ and $\int_{\partial \Omega} \phi\, d\mathcal{H}^{n-1}=a$. Using \eqref{fuerteborde}, we get $\|u\|_{L^p(\partial\Omega)}=1$. As a consequence of the lower semicontinuity of the norm $\|.\|_{W^{1,p} (\Omega)}$ with respect to weak convergence, we obtain
\begin{equation}\label{normap}
\int_{\Omega}|\nabla u|^p+|u|^pdx \leq \liminf\limits_{k\rightarrow
\infty} \int_{\Omega}|\nabla u_k|^p+|u_k|^pdx
\end{equation}
Using \eqref{fuerteborde}, we can see that  $ |u_k|^p\rightarrow |u|^p \quad\hbox{in }L^{1} (\partial \Omega)$. Therefore, taking into account \eqref{debil*} we obtain
\begin{equation}\label{fixup}
\int_{\partial \Omega}\phi_k |u_k|^p d\mathcal{H}^{n-1}\rightarrow \int_{\partial \Omega}\phi |u|^p d\mathcal{H}^{n-1}
\end{equation}
From \eqref{normap} and \eqref{fixup}, we have $(u,\phi)$ is an optimal pair for \eqref{optimo}.

By an elementary variation of the \emph{Bathtub Principle} (\cite[Pag. 28]{Lieberman}), we can prove that the minimization problem
\[
 \inf\limits_{\int_{\partial\Omega}\phi d\mathcal{H}^{n-1}=a}\int_{\partial\Omega} \phi|u|^pd\mathcal{H}^{n-1},
\]
has a solution of the form $\phi=\chi_D$, where $\{u<s\}\subset D\subset\{u\leq s\}$ and $\mathcal{H}^{n-1}(D)=a$ and therefore $(\chi_D, u)$ is an optimal pair for $\Lambda(\sigma, a)$.
\end{proof}

Now we prove a Lemma about the continuity of the eigenvalues and eigenfunctions with respect to the potential $\phi$ in the weak * topology.

\begin{lem}\label{continuidad.u}
Let $\phi_j,\phi\in L^\infty(\partial\Omega)$ be such that $\phi_j\stackrel{*}{\rightharpoonup} \phi$ in $L^\infty(\partial\Omega)$. Let $\lambda_j = \lambda(\sigma, \phi_j)$ and $\lambda = \lambda(\sigma, \phi)$ the eigenvalues defined by  \eqref{var prob} and let $u_j, u\in W^{1,p}(\Omega)$ be the positive normalized eigenfunctions associated to $\lambda_j$ and $\lambda$ respectively.

Then $\lambda_j\to \lambda$ and $u_j\to u$ strongly in $W^{1,p}(\Omega)$ as $j\to\infty$.
\end{lem}

\begin{proof}
First, define $v\equiv  \mathcal{H}^{n-1}(\partial\Omega)^{-1/p}$ and from \eqref{var prob} we get
$$
\lambda_j\le I(v, \phi_j) = \frac{|\Omega| + \int_\Omega \phi_j}{\mathcal{H}^{n-1}(\partial\Omega)}\le C
$$
for every $j\in\N$. Therefore, since $\|u_j\|_{W^{1,p}(\Omega)}\le \lambda_j$ (recall that the eigenfunctions $u_j$ are normalized) it follows that $\{u_j\}_{j\in\N}$ is bounded in $W^{1,p}(\Omega)$.

From these, we obtain the existence of a function $w\in \sob$ such that, for a subsequence,
\begin{align*}
&u_j\rightharpoonup w \qquad \text{weakly in } \sob\\
&u_j\to w \qquad \text{strongly in } L^p(\Omega)\\
&u_j\to w \qquad \text{strongly in } L^p(\partial\Omega)
\end{align*}
It then follows that $w\ge 0$ and that $\|w\|_{L^p(\partial\Omega)} = 1$. 

Now, from the weakly sequentially lower semicontinuity it holds
\begin{equation}\label{w=u}
\lambda\le I(w,\phi)\le \liminf I(u_j, \phi) = \liminf I(u_j, \phi_j) + \sigma \int_{\partial\Omega} (\phi-\phi_j)|u_j|^p\, d\mathcal{H}^{n-1}.
\end{equation}
Since $|u_j|^p\to |u|^p$ strongly in $L^1(\partial\Omega)$, it easily follows that
$$
\lambda\le \liminf \lambda_j.
$$

For the reverse inequality, we proceed in a similar fashion. In fact, from \eqref{var prob}
$$
\lambda_j \le I(u, \phi_j).
$$
Therefore
$$
\limsup \lambda_j \le \lim I(u, \phi_j) = I(u, \phi)=\lambda,
$$
so $\lambda_j\to \lambda$.

Finally, from \eqref{w=u}, one obtains that $I(w, \phi)=\lambda$ and since there exists a unique nonnegative normalized eigenfunction associated to $\lambda$ it follows that $w=u$. Moreover, again from \eqref{w=u} it is easily seen that $\|u_j\|_{\sob}\to\|u\|_{\sob}$ and so $u_j\to u$ strongly in $\sob$ and, since the limit is uniquely determined, the whole sequence $\{u_j\}_{j\in\N}$ is convergent.
\end{proof}

The next Lemma, that was proved in \cite{DPFBN} gives the strict monotonicity of the quantity $\Lambda(\infty, a)$ with respect to $a$ and will be helpful in showing the behavior of $\Lambda(\sigma, a)$ for $\sigma\to\infty$.

\begin{lem}[Corollary 3.7, \cite{DPFBN}]\label{monot}
The function $\Lambda (\infty,\cdot)$ is strictly monotonic.
\end{lem}

Now we are ready to prove the convergence of $\Lambda(\sigma, a)$ to $\Lambda(\infty, a)$ as $\sigma\to\infty$.

\begin{thm} If $\sigma_j$ is a sequence tending to $\infty$ and $(u_j,D_j)$ associated optimal pairs of \eqref{optimo}, then there
exists a subsequence (that we still call $\sigma_j$) and an optimal pair $(u,D)$ of the problem \eqref{agujerooptimo} such that
$u_j\rightharpoonup u$ in $\sob$, $\chi_{D_j}\overset{*}{\rightharpoonup}\chi_D$ in $L^{\infty}(\partial\Omega)$. 
\end{thm}

\begin{proof}
We consider $E\subset\partial\Omega$ closed such that $\mathcal{H}^{n-1}(E)=a$ and $v\in \sob$, $\|v\|_{L^p(\partial\Omega)}=1$ such that $ v=0$ in $E$. Therefore
\[\|u_j\|^p_{\sob} \leq I(u_j, \chi_{D_j}) = \Lambda (\sigma_j,a) \leq \lambda(\sigma_j, \chi_E) \le I(v, \chi_E)=\|v\|^p_{\sob}\]
Hence, the sequence $u_j$ is bounded in $\sob$. Therefore we can assume that there exists $u_{\infty} \in \sob$ and
$\phi_{\infty}\in L^{\infty}(\partial\Omega)$ such that
\begin{align}
\label{conv1}u_j & \rightharpoonup u_{\infty} \hbox{ in } \sob\\
\label{conv2} u_j &\rightarrow u_{\infty}  \hbox{ in }    L^p(\Omega)\\
\label{conv3} u_j  &\rightarrow u_{\infty}  \hbox{ in }   L^p(\partial\Omega)\\
\label{conv4} \chi_{D_j} & \overset{*}{\rightharpoonup}\phi_{\infty} \hbox{ in }   L^{\infty}(\partial\Omega)
\end{align}
From \eqref{conv3} and \eqref{conv4} we have that $\|u_{\infty}\|_{L^p(\partial\Omega)}=1$,
$\int_{\partial\Omega}\phi_{\infty}d\mathcal{H}^{n-1}=a$ and $0\leq\phi_{\infty}\leq 1$. The rest of the proof follows in a
completely analogous way, using Lemma \ref{monot},  to \cite[Theorem 1.2]{DPFBR}
\end{proof}

\section{Symmetry}
Throughout this section we assume that $\Omega$ is the unit ball $B(0,1)$. The goal of the section is to show that there exists an
optimal pair $(u, \chi_D)$ of the problem \eqref{pde} with $D$ a spherical cup in $S^{n-1} = \partial\Omega$. A key tool is played by the
\emph{spherical symmetrization}.

The spherical symmetrization of a set $A \subset \mathbb{R}^n$ with respect to an axis given by a unit vector $e$ is defined as follows: Given $r>0$ we consider $s_r>0$ such that
$\mathcal{H}^{n-1}(A\cap\partial B(0,r))=\mathcal{H}^{n-1}(B(re,s_r)\cap\partial B(0,r))$. We note that the sets $A\cap\partial
B(0,r)$ are $\mathcal{H}^{n-1}$-measurable for almost every $r\geq 0$.  Now we put:

 \[A^{*}=\bigcup_{0\leq r\leq 1}B(re,s_r)\cap\partial B(0,r)\]

The set $A^*$ is well defined and measurable whence $A$ is a measurable set. If $u\geq 0$ is a measurable function, we define its symmetrized function $u^*$ so that satisfies the relation $\{u^*\geq t\}=\{u\geq t\}^*$. We refer to \cite{Kawohl} for an exhaustive study of this symmetrization. In particular, we need the following known results:
\begin{thm}\label{propiedades}
Let $0\leq u \in W^{1,p}(\Omega)$ and let $u*$ be its symmetrized function. Then
\begin{enumerate}
\item $u^* \in W^{1,p}(\Omega)$
\item $u^*$ and $u$ are equi-measurable, i.e. they have the same distribution function, Hence for every continuos increasing function
$\Phi$:  $\int_{\Omega} \Phi(u^*)dx=\int_{\Omega} \Phi(u)dx$
\item $\int_{\Omega} uv dx\leq\int_{\Omega} u^*v^*dx$, for every measurable positive function $v$.
\item In a similar way $u$ and $u^*$ are equimeasurable respect to the Hausdorff measure on boundary of balls. Therefore, the two
previous items holds with $\partial \Omega$ and $d\mathcal{H}^{n-1}$ instead of $\Omega$ and $dx$, respectively.
\item $\int_{\Omega}|\nabla u^*|^pdx\leq \int_{\Omega}|\nabla u|^pdx$.
\end{enumerate}
\end{thm}

With these preliminaries, we can now prove the main result of the section.

\begin{thm} Let $\Omega=B(0,1)$. Then there exists an optimal pair $(u, \chi_E)$ of the problem \eqref{pde} with $E$ a spherical cup in $\partial\Omega$.

\end{thm}
\begin{proof}
Let $(u, \chi_D)$ be an optimal pair. We define $E:=((D^c)^*)^c$. Since $(D^c)^*$ is a spherical cup it follows that $E$ is also a spherical cup.  

We note that $\chi_E=1-(\chi_{D^c})^*$, therefore it is easy to show, from (c) in Theorem \ref{propiedades} that
\[\int_{\partial\Omega} \chi_E |u^*|^p d\mathcal{H}^{n-1}\leq \int_{\partial\Omega} \chi_D |u|^p d\mathcal{H}^{n-1}.\]
We note that $\int_{\partial\Omega} |u^*|^pd\mathcal{H}^{n-1}=1$, so $u^*$ is an admissible function in \eqref{var prob}.
Moreover,
\[\int_{\Omega}|\nabla u^*|^p+|u^*|^pdx+\sigma\int_{\partial \Omega}\chi_E|u^*|^pd\mathcal{H}^{n-1}\leq \int_{\Omega}|\nabla u|^p+|u|^pdx+\sigma\int_{\partial
\Omega}\chi_D|u|^pd\mathcal{H}^{n-1}.\] Consequently, $(u^*, \chi_E)$ is an optimal pair.
\end{proof}

\section{Derivative of Eigenvalues}
Henceforth we put $\Gamma:=\partial\Omega$. In this section we compute derivatives of the eigenvalues $\lambda(\sigma, \chi_D)$ with respect to perturbations of the set $D$. We also assume that the set $D\subset \Gamma$ is the closure of a regular relatively open set.

For this purpose, we introduce the vector field
$V:\mathbb{R}^n\rightarrow\mathbb{R}^n$ supported on a narrow neighborhood of $\Gamma$ with $V\cdot\n=0$, where $\n$ is the
outer normal vector. We consider the flow
\begin{equation}\label{flujo}
    \left\{%
\begin{array}{rl}
    \frac{d}{dt}\Psi_t(x)&=V(\Psi_t(x)) \\
    \Psi_0(x)&=x \\
\end{array}%
\right.
\end{equation}
We note that the condition $V\cdot\n=0$ implies that $\Psi_t(\Gamma)=\Gamma$. From \eqref{flujo}, it follows the asymptotic expansions
\begin{align}
D\Psi_t&=I+tDV+o(t),\\
(D\Psi_t)^{-1}&=I-tDV+o(t),\\
J\Psi_t&=1+t\dive V+o(t).
\end{align}
Here $D\Psi_t$ and $J\Psi_t$ denote the differential matrix of $\Psi_t$ and its jacobian, respectively. See \cite{Henrot-Pierre}.

In order to try with surface integrals, we need the following formulas whose proofs can be founded in \cite{Henrot-Pierre}. The tangential Jacobian of $\Psi_t$ is given by
\[ J_{\Gamma}\Psi_t(x)=|(D\Psi(x))^{-1}\n|J\Psi(x)=1+t\dive_{\Gamma}V+o(t)\quad x\in\Gamma\]
where $\dive_{\Gamma}V$ is the tangential divergence operator defined by 
$$
\dive_{\Gamma}V=\dive V-\n^T DV\n.
$$

The main result here is the following

\begin{thm}\label{teo.dif}
Let $\sigma>0$ be fixed and $D\subset\Gamma$ be the closure of a smooth relatively open set. Let $u \in W^{1,p}(\Omega)$ be the nonnegative normalized eigenfunction for $\lambda(\sigma,\chi_D)$. 

Then, the function $\lambda(t):=\lambda(\sigma, \chi_{D_t})$ where $D_t=\Psi_t(D)$ is differentiable at $t=0$ and 
\[\lambda'(0) = -\sigma\int_{\partial_{\Gamma} D} |u_0|^p (\n_{\Gamma}\cdot V)d\mathcal{H}^{n-2}\] 
where $\n_{\Gamma}$ denotes the unit normal vector  exterior to
$\partial_{\Gamma}D$ relative to the tangent space of $\Gamma$.
\end{thm}

\begin{rem}\label{rem.cont}
Observe that the results of Lemma \ref{continuidad.u} immediately imply the continuity of $\lambda(t)$ at $t=0$ and also that the associated eigenfunctions $u_t$ strongly converge to the associated eigenfunction $u$ of $\lambda(0)$ in $\sob$.
\end{rem}

\begin{proof}[Proof of Theorem \ref{teo.dif}]We will follow the same line that \cite[Theorem 1.1]{DP}. Let $u \in W^{1,p}(\Omega)$.
We call $\overline{u}=u\circ\Psi_t$, then the following asymptotic expansions hold
\begin{equation}\label{primera}
\begin{split} &\int_{\Omega}|\nabla \bar u|^p+|\bar u|^p dx=\int_{\Omega}\big(|D\Psi_t\nabla
u|^p+|u|^p\big)J\Psi_t^{-1} dx\\ &=\int_{\Omega}\big(|(I+tDV+o(t))\nabla u|^p+|u|^p\big)(1-t\dive V+o(t)) dx\\
&=\int_{\Omega}|\nabla u|^p+|u|^p dx - t \big(\dive V(|\nabla u|^p + |u|^p) - p |\nabla
u|^{p-2} (\nabla u)^tDV\nabla u\big) dx + o(t)
\end{split}\end{equation}

\begin{equation}
\begin{split}\label{segunda}
\int_{\Gamma}\chi_{D_t}|\bar u|^p d\mathcal{H}^{n-1}&=\int_{\Gamma}  \chi_D |u|^p J_{\Gamma}\Psi_t^{-1}\, d\mathcal{H}^{n-1}\\
&=\int_{\Gamma} \chi_D |u|^p (1-t\dive_{\Gamma}V)d\mathcal{H}^{n-1}+o(t)
\end{split}\end{equation}

\begin{equation}
\begin{split}\label{tercera}
\int_{\Gamma}|\bar u|^p d\mathcal{H}^{n-1}&=\int_{\Gamma}  |u|^p J_{\Gamma}\Psi_t^{-1}
d\mathcal{H}^{n-1}\\
&=\int_{\Gamma} |u|^p (1 - t\dive_{\Gamma}V)d\mathcal{H}^{n-1}+o(t)
\end{split}\end{equation}
From \eqref{primera} and \eqref{segunda}, we obtain
\begin{equation}
I(\bar u,\chi_{D_t})= F(u) - tG(u) + o(t)
\end{equation}
where
\begin{equation}
F(u)=\int_{\Omega}|\nabla u|^p+|u|^p dx + \sigma \int_{\Gamma} \chi_D |u|^p\, d\mathcal{H}^{n-1}
\end{equation} 
and
\begin{equation}
G(u)=\int_{\Omega}\dive V(|\nabla u|^p + |u|^p) - p |\nabla u|^{p-2} (\nabla u)^t DV \nabla u\, dx + \sigma \int_{\Gamma} \chi_D |u|^p \dive_{\Gamma}V\,  d\mathcal{H}^{n-1}
\end{equation}

Now, take $u$ to be a normalized eigenfunction associated to $\lambda(0)$. Then we have
\begin{align*}
\lambda(t)&\le \frac{I(\bar u, \chi_{D_t})}{\int_\Gamma |\bar u|^p\, d\mathcal{H}^{n-1}} = \frac{F(u) - t G(u) + o(t)}{\int_\Gamma |u|^p\, d\mathcal{H}^{n-1} - t \int_\Gamma |u|^p \dive V d\mathcal{H}^{n-1} + o(t)}\\
& = \frac{F(u)}{\int_\Gamma |u|^p\, d\mathcal{H}^{n-1}} + t \left(F(u) \frac{\int_\Gamma |u|^p \dive V\, d\mathcal{H}^{n-1}}{\Big(\int_\Gamma |u|^p\, d\mathcal{H}^{n-1}\Big)^{2} } - \frac{G(u)}{\int_\Gamma |u|^p\, d\mathcal{H}^{n-1}}\right) + o(t)\\
&= \lambda(0) + t \left(\lambda(0)\int_\Gamma |u|^p \dive V\, d\mathcal{H}^{n-1} - G(u)\right) + o(t)
\end{align*}

Therefore
\begin{equation}\label{desig1}
\lambda(t) - \lambda(0) \le   t \left(\lambda(0)\int_\Gamma |u|^p \dive V\, d\mathcal{H}^{n-1} - G(u)\right) + o(t)
\end{equation}

Now, take $u_t\in W^{1,p}(\Omega)$ a normalized eigenfunction associated to $\lambda(t)$ and denote by $\bar u_t = u_t\circ \Psi_{-t}$. So
\begin{align*}
\lambda(0) &\le \frac{I(\bar u_t, \chi_D)}{\int_\Gamma |\bar u_t|^p\, d\mathcal{H}^{n-1}} = \frac{F(u_t) + t G(u_t) + o(t)}{\int_\Gamma |u_t|^p\, d\mathcal{H}^{n-1} + t\int_\Gamma |u_t|^p \dive_\Gamma V\, d\mathcal{H}^{n-1} + o(t)}\\
&= \frac{F(u_t)}{\int_\Gamma |u_t|^p\, d\mathcal{H}^{n-1}} - t \left(F(u_t) \frac{\int_\Gamma |u_t|^p \dive V\, d\mathcal{H}^{n-1}}{\Big(\int_\Gamma |u_t|^p\, d\mathcal{H}^{n-1}\Big)^{2} } - \frac{G(u_t)}{\int_\Gamma |u_t|^p\, d\mathcal{H}^{n-1}}\right) + o(t)\\
&= \lambda(t) - t \left(\lambda(t)\int_\Gamma |u_t|^p \dive V\, d\mathcal{H}^{n-1} - G(u_t)\right) + o(t)
\end{align*}

This last inequality together with \eqref{desig1} give us
\begin{align*}
t \Big(\lambda(t)\int_\Gamma &|u_t|^p \dive V\, d\mathcal{H}^{n-1} - G(u_t)\Big) + o(t)  \le \lambda(t) - \lambda(0)\\
&  \le   t \left(\lambda(0)\int_\Gamma |u|^p \dive V\, d\mathcal{H}^{n-1} - G(u)\right) + o(t)
\end{align*} 
So, by Remark \ref{rem.cont} one gets
$$
\lambda'(0) = \left(\lambda(0)\int_\Gamma |u|^p \dive V\, d\mathcal{H}^{n-1} - G(u)\right).
$$

It remains to further simplify the expression for $\lambda'(0)$. Let
\[\begin{split}G(u)&= \int_{\Omega}\dive V(|\nabla u|^p+|u|^p)-p|\nabla u|^{p-2}(\nabla u)^t DV \nabla u\, dx\\
&+\sigma\int_{\Gamma} \chi_D |u|^p\dive_{\Gamma}V\, d\mathcal{H}^{n-1}\\ 
&=I_1+I_2
\end{split}
\]

Now using $V\cdot\nabla u$ as test function in the equation $-\Delta_p u + |u|^{p-2}u=0$ and the boundary condition in \eqref{pde}  we obtain:
\[
\begin{split}
  I_1&=-p\int_{\Gamma} |\nabla u|^{p-2} \frac{\partial u}{\partial\n}V\cdot\nabla u\, d\mathcal{H}^{n-1}\\
&=-p \int_{\Gamma}\lambda(0) |u|^{p-2}u(V\cdot\nabla u) - \sigma\chi_D|u|^{p-2}u(V\cdot\nabla u)d\mathcal{H}^{n-1}\\
&= -\lambda(0) \int_\Gamma \nabla(|u|^p) \cdot V\, d\mathcal{H}^{n-1} + \sigma\int_\Gamma \chi_D\nabla(|u|^p)\cdot V\, d\mathcal{H}^{n-1}\\
&= \lambda(0)\int_\Gamma |u|^p \dive_\Gamma V\, d\mathcal{H}^{n-1} - \sigma\int_D |u|^p \dive_\Gamma V\, d\mathcal{H}^{n-1} + \sigma \int_{\partial_\Gamma D} |u|^p V\cdot \n_\Gamma\, d\mathcal{H}^{n-2}\\
&= \lambda(0)\int_\Gamma |u|^p \dive_\Gamma V\, d\mathcal{H}^{n-1} + \sigma\int_{\partial_\Gamma D} |u|^p V\cdot \n_\Gamma\, d\mathcal{H}^{n-2} - I_2
\end{split}
\]

So, 
$$
G(u) = \lambda(0)\int_\Gamma |u|^p \dive_\Gamma V\, d\mathcal{H}^{n-1} + \sigma\int_{\partial_\Gamma D} |u|^p V\cdot \n_\Gamma\, d\mathcal{H}^{n-2}
$$
and therefore
$$
\lambda'(0) = - \sigma\int_{\partial_\Gamma D} |u|^p V\cdot \n_\Gamma\, d\mathcal{H}^{n-2}
$$
This completes the proof of the Theorem.
\end{proof}

\section*{Acknowledgements}
This work was partially supported by Universidad de Buenos Aires under grant UBACYT 20020100100400 and by CONICET (Argentina) PIP 5478/1438.

\bibliographystyle{plain}
\bibliography{biblio}

\end{document}